\documentclass[11pt,a4paper]{article}
\usepackage{epsfig}
\usepackage{graphics}
\usepackage{subfigure}
\usepackage[mac]{inputenc}
\usepackage{epstopdf}
\usepackage{hyperref}

\usepackage[english]{babel}
\usepackage{mdwlist}
\usepackage{graphicx}
\usepackage{appendix}
\usepackage{dsfont}
\usepackage{tikz}
\usepackage{amsmath}
\usepackage{amsbsy}
\usepackage{amsfonts}
\usepackage{amssymb}
\usepackage{amscd}
\usepackage{amsthm}

\theoremstyle{definition}
\newtheorem{defi}{Definition}[section]
\newtheorem{prop}{Proposition}[section]

\newtheorem{hyp}{Assumption}[section]
\newtheorem{thm}{Theorem}[section]

\numberwithin{equation}{section}

\newcommand{\B}{\mathcal{B}}
\newcommand{\C}{\mathcal{C}}

\newcommand{\bigO}[1]{\ensuremath{\mathop{}\mathopen{}O\mathopen{}\left(#1\right)}}
\newcommand{\smallO}[1]{\ensuremath{\mathop{}\mathopen{}o\mathopen{}\left(#1\right)}}

\newcommand{\smallOp}[1]{\ensuremath{\mathop{}\mathopen{}o_{p}\mathopen{}\left(#1\right)}}

\newcommand{\CC}{\mathbb{C}}
\newcommand{\Card}{\mathrm{Card}}

\renewcommand{\dim}{\mathrm{dim}}

\newcommand{\EE}{\mathbb{E}}

\renewcommand{\H}{\mathcal{H}}
\newcommand{\Id}{\mathrm{Id}}

\newcommand{\LL}{\mathbb{L}}
\newcommand{\M}{\mathcal{M}}

\newcommand{\NN}{\mathbb{N}}

\newcommand{\PP}{\mathbb{P}}

\newcommand{\R}{\mathcal{R}_{n}}

\newcommand{\RR}{\mathbb{R}}

\newcommand{\TT}{\mathbb{T}}
\newcommand{\Tr}{\mathrm{Tr}\,}

\newcommand{\ZZ}{\mathbb{Z}}
\newcommand{\ud}{\,\mathrm{d}}

\makeatletter
\newcommand{\bigint}{\@ifnextchar_\@bigintsub\@bigintnosub}
\def\@bigintsub_#1{\def\@int@subscript{#1}\@ifnextchar^\@bigintsubsup\@bigintsubnosup}
\def\@bigintsubsup^#1{\mathop{\text{\large$\int_{\text{\normalsize$\scriptstyle\@int@subscript$}}^{\text{\normalsize$\scriptstyle#1$}}$}}\nolimits}
\def\@bigintsubnosup{\mathop{\text{\large$\int_{\text{\normalsize$\scriptstyle\@int@subscript$}}$}}\nolimits}
\def\@bigintnosub{\@ifnextchar^\@bigintnosubsup\@bigintnosubnosup}
\def\@bigintnosubsup^#1{\mathop{\text{\large$\int^{\text{\normalsize$\scriptstyle#1$}}$}}\nolimits}
\def\@bigintnosubnosup{\mathop{\text{\large$\int$}}\nolimits}
\makeatother

\newcommand{\Var}{\text{Var}}

\DeclareMathOperator*{\argm}{argmin}

\newcommand{\norme}[1]{\left\Vert #1 \right\Vert }

\makeatother

\title{Random action of compact Lie groups and minimax estimation of a mean pattern}

\author{Jérémie Bigot, Claire Christophe and Sébastien Gadat \vspace{0.2cm}  \\
Institut de Math\'ematiques de Toulouse\\
Universit\'e de Toulouse et CNRS (UMR 5219)\\
31062 Toulouse, Cedex 9, France\\
{\small {\tt \{Jeremie.Bigot,  Claire.Christophe, Sebastien.Gadat\}@math.univ-toulouse.fr }}
 \vspace{0.2cm}}

\begin{document}
\maketitle

\begin{abstract}
This paper considers the problem of estimating a mean pattern in the setting of Grenander's pattern theory. Shape variability in a data set of curves or images is modeled by the random action of elements in a compact Lie group on an infinite dimensional space. In the case of observations contaminated by an additive Gaussian white noise, it is shown that  estimating a reference template in the setting of Grenander’s pattern theory falls into the category of deconvolution problems over Lie groups. To obtain this result, we build an estimator of a mean pattern by using  Fourier deconvolution and harmonic analysis on compact Lie groups.  In an asymptotic setting where the number of observed curves or images tends to infinity, we derive upper and lower bounds for the minimax quadratic risk over Sobolev balls. This rate depends on the smoothness of the density of the random Lie group elements representing shape variability in the data, which makes a  connection between estimating a mean pattern and standard deconvolution problems in nonparametric statistics.
\end{abstract}



\subsubsection*{Acknowledgements}

The authors acknowledge the support of the French Agence Nationale de la Recherche (ANR) under reference ANR-JCJC-SIMI1 DEMOS.

\section{Introduction}

In signal and image processing, data are often in the form of a set of $n$ curves or images $Y_{1},\ldots,Y_{n}$. In many applications, observed curves or images have a similar structure which may lead to the assumption that these observations are random elements which vary around the same mean pattern (also called reference template). However, due to additive noise and shape variability in the data, this mean pattern is typically unknown and has to be estimated.  In this setting, a widely used approach is Grenander’s pattern theory \cite{Gre,gremil} which models shape variability by the action of a Lie group on an infinite dimensional space of curves or images.  In the last decade, the study of  transformation Lie groups  to model shape variability of images has been an active research field, and we refer to \cite{MR2176922,TYShape} for a recent overview of the theory of deformable templates. Currently, there is also a growing interest in statistics on the problem of estimating the mean pattern of a set of curves or images using deformable templates \cite{allason,MR2730643,BG10,BGL09,BGV09,mmty07}. In this paper, we focus on the problem of constructing asymptotically minimax estimators of a mean pattern using noncommutative Lie groups  to model shape variability. The main goal of this paper is to show that estimating a reference template in the setting of Grenander’s pattern theory falls into the category of deconvolution problems over Lie groups as formulated in \cite{MR2446754}.

To be more precise, let $G$ be a connected and compact Lie group. Let $ \LL^2(G)$ be the Hilbert space of complex valued, square integrable functions on the group G with respect to the Haar measure $\ud g$. We propose to study the nonparametric estimation of a complex valued function $f^{\star} : G \to \CC$ in the following deformable white noise model
\begin{eqnarray}\label{modele4}
\ud Y_m(g) &=& f_m(g)\ud g+\varepsilon\ud W_m(g), \quad g\in G,\ m \in[\![1,n]\!]
\end{eqnarray}
where
$$
f_m(g)=f^{\star}(\tau_m^{-1}g).
$$
The $\tau_m$'s are independent and identically distributed (i.i.d) random variables belonging to $G$ and the $W_m$'s are independent standard Brownian sheets on the topological space $G$ with reference measure $\ud g$. For all $m = 1,\ldots,n$, $\tau_m$ is also supposed to be independent of $W_m$. 

In  \eqref{modele4} the function $f^{\star}$ is the unknown mean pattern to estimate in the asymptotic setting $n \to + \infty$, and $ \LL^2(G)$ represents an infinite dimensional space of curves or images. The $\tau_m$'s are random variables acting on  $ \LL^2(G)$  and they model shape variability in he data. The $W_{m}$ model intensity variability in the observed curves or images. In what follows, the random variables $\tau_m$ are also supposed to have a known density $h \in \LL^2(G)$. We will show that  $h$ plays the role of the kernel a convolution operator that has to be inverted to construct an optimal (in the minimax sense) estimator of $f^{\star}$. Indeed, since $W_{m}$ has zero expectation, it follows that the expectation of the $m$-th observation in \eqref{modele4} is equal to
$$
\EE f_m(g) =  \int_{G}  f^{\star}(\tau^{-1}g) h(\tau) \ud  \tau \mbox{ for any } m \in[\![1,n]\!].
$$
Therefore, $\EE f_m(g) = f^{\star} * h$ is the convolution over the group $G$ between the function $f^{\star}$ and the density $h$. Hence, we  propose to build an estimator of $f^{\star}$ using a regularized deconvolution method over Lie groups. This class of inverse problems is based on the use of harmonic analysis and Fourier analysis on compact Lie groups to transform convolution in a product of Fourier coefficients. However, unlike standard Fourier deconvolution on the torus, when  $G$ is not a commutative group,  the Fourier coefficients of a function in $\LL^2(G)$ are no longer complex coefficients but grow in dimension with increasing ``frequency". This somewhat complicates both the inversion process and the study of the asymptotic minimax properties of the resulting estimators.

In \cite{BLV09}, a model similar to  \eqref{modele4} has been studied where $n$ is held fixed, and the $\tau_m$'s are not random but deterministic parameters to be estimated in the asymptotic setting $\epsilon \to 0$ using semi-parametric statistics techniques. The potential of  using noncommutative harmonic analysis  for various applications  in  engineering is well described in \cite{MR1885369}. The contribution of this paper is thus part of the growing interest  in nonparametric statistics and inverse problems on the use of harmonic analysis on Lie groups \cite{MR1635446,MR1889831,MR2446754,MR1873674,MR2772533,MR2594372,MR2045023}.  

Our construction of an estimator of the mean pattern in \eqref{modele4} is inspired by the following problem of stochastic deconvolution over Lie groups introduced in  \cite{MR2446754}: estimate $f^{\star} \in \LL^2(G)$ from the regression model
\begin{equation}
y_{j} =  \int_{G}  f^{\star}(\tau^{-1}g_{j}) h(\tau) \ud  \tau + \eta_{j}, \quad g_{j}\in G, \; j \in[\![1,n]\!] \label{modelKK}
\end{equation}
where $h$ is a known convolution kernel, the $g_{j}$'s are ``design points'' in $G$,  and the $\eta_{j}$'s are independent realizations of a random noise process with zero mean and finite variance.  In \cite{MR2446754} a notion of asymptotic minimaxity over $\LL^2(G)$ is introduced, and the authors derive upper and lower bounds for a minimax risk over Sobolev balls. In this paper we also introduce a notion of minimax risk in model  \eqref{modele4}. However, deriving upper and lower bounds of the minimax risk for the estimation of $f^{\star}$  is significantly more difficult  in \eqref{modele4} than in model \eqref{modelKK}. This is due to the fact that there are two sources of noise in model \eqref{modele4}: a source of  additive Gaussian noise $W_m$ which is a classical one for studying minimax properties of an estimator, and a source of shape variability due to the $\tau_m$'s which is much more difficult to treat. In particular, standard methods to derive lower bounds of the minimax risk in classical white noise models  such as  Fano's Lemma are not adapted to the source of shape variability in  \eqref{modele4}. We show that one may use the Assouad’s cube technique (see e.g.\ \cite{MR2724359} and references therein), but  it has to be carefully adapted to  model \eqref{modele4}.

The paper is organized as follows. In Section \ref{sec:estim}, we describe the construction of our estimator using a deconvolution step and Fourier analysis on compact Lie groups. We also define a notion of asymptotic optimality in the minimax sense for estimators of the mean pattern. In Section \ref{sec:bounds}, we derive an upper bound on the minimax risk that depends on smoothness assumptions on the density $h$. A lower bound  on the minimax risk is also given. All proofs are gathered in a technical appendix. At the end of the paper, we have also included some technical materials about Fourier analysis on compact Lie groups, along with some formula for the rate of convergence of the eigenvalues of the Laplace-Beltrami operator which are needed to derive our asymptotic rates of convergence.  

\section{Mean pattern estimation via deconvolution on Lie groups} \label{sec:estim}

In this section, we use various concepts from harmonic analysis on Lie groups which are defined in Appendix \ref{app:Liegroup}.

\subsection{Sobolev space in $\LL^2(G)$}

Let $\widehat{G}$ be the set of equivalence classes of irreducible representations of $G$ that is  identified  to the set of unitary representations of each class. For   $\pi\in\widehat{G}$ and $g\in G$ one has that $\pi(g)\in GL_{d_{\pi}\times d_{\pi}}(\CC)$ (the set of $d_{\pi} \times d_{\pi}$ nonsingular matrices with complex entries) where $d_{\pi}$ is the dimension of $\pi$. By the Peter-Weyl theorem (see Appendix \ref{PWpartie}), any function $f\in\LL^2(G)$ can be decomposed as
\begin{equation}
f(g)=\sum_{\pi\in \widehat{G}}d_{\pi}\Tr\left(\pi(g)c_{\pi}(f)\right), \label{eq:decomp}
\end{equation}
where $\Tr$ is the trace operator and $c_{\pi}(f) =\int_G f(g)\pi(g^{-1}) \ud g$ is the $\pi$-th Fourier coefficient of $f$ (a $d_{\pi} \times d_{\pi}$ matrix). The decomposition formula \eqref{eq:decomp} is an analogue of the usual Fourier analysis in  $\LL^2([0,1])$ which corresponds to the situation $G = \RR/\ZZ$ (the torus in dimension 1) for which $\widehat{G} = \ZZ$, the representations $\pi$ are the usual trigonometric polynomials $\pi(g) = e^{i 2 \boldsymbol{\pi} \ell g}$ for some $\ell \in \ZZ$ (with the bold symbol $\boldsymbol{\pi}$ denoting the number Pi). In this case, the matrices $c_{\pi}(f)$ are one-dimensional ($d_{\pi} = 1$) and they equal the standard Fourier coefficients $c_{\pi}(f) = c_{\ell}(f)  =\int_{0}^{1} f(g) e^{- i 2 \boldsymbol{\pi} \ell g} \ud g$. For  $G = \RR/\ZZ$, one thus retrieves the classical Fourier decomposition of a periodic function $f : [0,1] \to \RR$ as $f(g)=\sum_{\ell \in \ZZ } c_{\ell}(f)  e^{i 2\boldsymbol{\pi} \ell g}$.


\begin{defi}
Let $k\in\NN^*$. Let $A \in \M_{k\times k}(\CC)$ (the set of $k \times k$ matrices with complex entries). The Frobenius norm of $A$ is defined by $ \norme{A}_F^2=\sqrt{\Tr\left(A\overline{A}^t\right)}$. It is the norm induced by the inner product $\langle A,B \rangle_{F}=\Tr(A \overline{B}^t)$ of two matrices $A,B\in\M_{k\times k}(\CC)$.
\end{defi}
By Parseval's relation, it follows that 
$
 ||f||^2  =  ||f||_{\LL^2(G)}^2 :=  \int_G |f(g)|^2\ud g =\sum_{\pi\in\widehat{G}}d_{\pi}\norme{c_{\pi}(f)}_F^2
$
for any $f \in \LL^2(G)$.
The following definitions of a Sobolev norm and Sobolev spaces have been proposed in \cite{MR2446754}.
\begin{defi}
Let $f\in\LL^2(G)$ and $s>\dim(G)/2 $. The Sobolev norm of order $s$ of  $f$ is defined by
$\|f\|_{H_s}^2 = \int_G |f(g)|^2\ud g + \sum_{\pi\in\widehat{G}}\lambda_{\pi}^s d_{\pi}\Tr\left(c_{\pi}(f)\overline{c_{\pi}(f)}^t\right)=\int_G |f(g)|^2\ud g + \sum_{\pi\in\widehat{G}}\lambda_{\pi}^s d_{\pi}\norme{c_{\pi}(f)}_F^2,$
where $\lambda_{\pi}$ is the eigenvalue value of $\pi$ associated to the Laplace-Beltrami operator induced by the Riemannian structure of the Lie group $G$.
\end{defi}

\begin{defi}
Let $s>\dim(G)/2 $ and denote by $C^{\infty}(G)$ the space of infinitely differentiable functions on $G$. The Sobolev space $H_s(G)$ of order $s$  is the completion of $C^{\infty}(G)$ with respect to the norm $\|\cdot\|_{H_s}$. Let $A>0$. The Sobolev ball of radius $A$ and order $S$ in $\LL^2(G)$ is defined as
$$
H_s(G,A)=\left\{f\in H_s(G)\ :\ \|f\|_{H_s}^2\leq A^2\right\}.
$$
\end{defi}
It can be checked that $H_s(G)$ corresponds to the usual notion of a Sobolev space in the case $G = \RR/\ZZ$. Now, let $\hat{f} \in \LL^2(G)$ be an estimator of $f^{\star}$ {\it i.e.} a measurable mapping of the random processes $Y_{m},m=1,\ldots,n$ taking its value in $\LL^2(G)$. The quadratic risk of an estimator $\hat{f} $ is defined as
$$
R(\hat{f},f^{\star})=\EE\left(\|\hat{f}-f^{\star}\|^2\right) = \EE \left(  \int_G |\hat{f}(g)-f^{\star}(g)|^2\ud g \right).
$$
\begin{defi}
The minimax risk over Sobolev balls associated to model \eqref{modele4} is defined as
$$
\R(A,s) = \inf_{\hat{f} \in\LL^2(G)}  \sup_{f^{\star} \in H_s(G,A)} R(\hat{f},f^{\star}),
$$
where the above infimum is taken over the set all estimators.
\end{defi}
The main goal of this paper is then to derive asymptotic upper and lower bounds on the minimax risk $\R(A,s)$ as $n \to + \infty$.

\subsection{Construction of the estimator}

First, note that the white noise model  \eqref{modele4}  has to be interpreted in the following sense: let $f \in \LL^2(G)$, then conditionally to $\tau_m$  each
integral $\int_{G} f(g) \ud Y_m(g) $ of the ``data'' $\ud Y_m(g)$ is a random variable normally distributed with mean $\int_{G} f(g) f^{\star}(\tau_m^{-1}g) \ud g $ and variance $\varepsilon^2 \int_{G} |f(g)|^2 dg$. Moreover, $\EE \left( \int_{G} f_1(g) \ud W_m(g)  \int_{G} f_2(g) \ud W_m(g) \right) = \int_{G} f_1(g) \overline{f_2} (g) \ud g $ for $f_1,f_2 \in \LL^2(G)$ and any $m \in[\![1,n]\!]$.
Therefore, using Fourier analysis on compact Lie groups, one may re-write model \eqref{modele4} in the Fourier domain as
\begin{eqnarray} \label{modele4transfome}
c_{\pi}(Y_m) &=& \int_G \pi(g^{-1})\ud Y_m (g)=c_{\pi}(f_m) +\varepsilon c_{\pi}(W_m), \; \mbox{ for } \pi\in\widehat{G} \mbox{ and } m\in[\![1,n]\!],
\end{eqnarray}
where
$$
c_{\pi}(f_m) = \int_G f_m(g)\pi(g^{-1}) \ud g \mbox{ and } c_{\pi}(W_m) = \int_G \pi(g^{-1})\ud W_m (g).
$$
Note that $c_{\pi}(f_m) = \int_G f^{\star}(\tau_m^{-1}g)\pi(g^{-1})\ud g= \int_G f^{\star}(g)\pi((\tau_m g)^{-1})\ud g$ which implies that
$$
c_{\pi}(f_m) = c_{\pi}(f^{\star})\pi(\tau_m^{-1}), \; m\in[\![1,n]\!].
$$
Remark also that the coefficients $(c_{\pi}(W_m))_{k,l}$ of the matrix $c_{\pi}(W_m)\in\M_{d_{\pi},d_{\pi}}(\CC)$ are independent complex random variables  that are normally distributed with zero expectation and variance $d_{\pi}^{-1}$. Moreover, note that
$$
\EE\left(\pi(\tau_m^{-1})\right)=c_{\pi}(h) \mbox{ and } \EE\left(c_{\pi}(Y_m)\right) =c_{\pi}(f^{\star})c_{\pi}(h).
$$
Therefore, if we assume that $c_{\pi}(h)$ is an invertible matrix, it follows that   an unbiased estimator of the  the $\pi$-th Fourier coefficient of $f^{\star}$ is given by the following deconvolution step in the Fourier domain 
\begin{equation}
\widehat{c_{\pi}(f^{\star})}=\frac{1}{n}\sum_{m=1}^n c_{\pi}(Y_m)c_{\pi}(h)^{-1}. \label{eq:deconv}
\end{equation}
An estimator of $f^{\star}$ can then be constructed by defining for $g \in G$
\begin{eqnarray}
 \hat{f}^{\star}_T(g) & = & \sum_{\pi\in \widehat{G}_T}d_{\pi}\Tr\left(\pi(g) \widehat{c_{\pi}(f^{\star})} \right)   \nonumber \\
& =& \frac{1}{n}\sum_{m=1}^n \sum_{\pi\in \widehat{G}_T}d_{\pi}\Tr\left(\pi(g)c_{\pi}(Y_m)c_{\pi}(h)^{-1}\right),  \label{estf2}
\end{eqnarray}
where $\widehat{G}_T=\left\{\pi\in\widehat{G}\ :\ \lambda_{\pi}<T\right\}$ for some $T>0$ whose choice has to be discussed (note that the cardinal of $\widehat{G}_T$ is finite).

\subsection{Regularity assumptions on the density $h$}

It is well-known that the difficulty of a deconvolution problem is quantified by the smoothness of the convolution kernel. The rate of  convergence that can be expected from any estimator depends on such smoothness assumptions. This issue has been  well studied in the nonparametric statistics literature on standard deconvolution problems (see e.g.\ \cite{MR1126324}). Following the approach proposed in \cite{MR2446754}, we now discuss a  smoothness assumption on the convolution kernel $h$.

\begin{defi}
Let $k\in\NN^*$ and $|.|_2$ be the standard Euclidean norm on $\CC^k$. The operator norm of   $A\in\M_{k\times k}(\CC)$ is 
$\|A\|_{op}= \sup_{u\neq 0}\frac{|Au|_2}{|u|_2}.$
\end{defi}

\begin{defi}
A function $f \in \LL^2(G)$ is said to be smooth of order  $\nu\geq 0$ if $c_{\pi}(f)$ is an invertible matrix for any $\pi \in \widehat{G}$, and if there exists two constants $C_1,\ C_2>0$ such that 
$$ \norme{c_{\pi}(f)^{-1}}_{op}^2\leq C_1 \lambda_{\pi}^{\nu} \mbox{ et } \|c_{\pi}(f)\|_{op}^2 \leq C_2 \lambda_{\pi}^{-\nu} \mbox{ for all } \pi\in\widehat{G}.$$
\end{defi}

\begin{hyp}\label{hregu}
The density $h$ is smooth of order $\nu\geq 0$.
\end{hyp}

Note that Assumption \ref{hregu} corresponds to the case where, in most applications, the convolution kernel $h$ leads to an inverse problem that is ill-posed, meaning in particular that there is no bounded inverse deconvolution kernel. This can be seen in the assumption $\norme{c_{\pi}(f)^{-1}}_{op}^2\leq C_1 \lambda_{\pi}^{\nu}$ which accounts for the setting where $\lim_{\lambda_{\pi} \to + \infty} \norme{c_{\pi}(f)^{-1}}_{op} = + \infty$ meaning that the mapping $f \mapsto f * h$ does not have a bounded inverse in $\LL^2(G)$. Example of such convolution kernels are discussed in \cite{MR1873674,MR2446754}, and we refer to these papers and references therein for specific examples.

\section{Upper and lower bounds} \label{sec:bounds}

The following theorem gives the asymptotic behavior of the quadratic risk of $\hat{f}_T$ over Sobolev balls using an appropriate choice for the regularization parameter $T$.

\begin{thm} \label{th:upperbound}
Suppose that Assumption \ref{hregu} holds. Let $\hat{f}_T$ be the estimator defined in \eqref{estf2} with $T = T_{n} =\lfloor n^{\frac{2}{2s+2\nu+\dim(G)}}\rfloor$. Let $s> 2 \nu + \dim(G)$. Then, there exists a constant $K_{1} >0$ such that
$$\limsup_{n\to\infty}\sup_{f^{\star}\in H_s(G,A)}n^{\frac{2s}{2s+2\nu+\dim(G)}}R(\hat{f}_{T_{n}},f^{\star})\leq K_{1}.$$
\end{thm}

Therefore, under Assumption \ref{hregu} on the density $h$, Theorem \ref{th:upperbound} shows that   the quadratic risk $R(\hat{f}_{T_{n}},f^{\star})$ is of polynomial order of the sample size $n$, and that this rate deteriorates as the smoothness $\nu$ of $h$ increases. The fact that estimating $f^{\star}$ becomes harder with larger of $\nu$ (the so-called degree of ill-posedness) is well known in standard deconvolution problems  (see e.g.\ \cite{MR1126324} and references therein). Hence, Theorem \ref{th:upperbound} shows that a similar phenomenon holds in model \eqref{modele4} when using the  deconvolution step \eqref{eq:deconv}.  The rate of convergence $n^{-\frac{2s}{2s+2\nu+\dim(G)}}$ corresponds to the minimax rate in model \eqref{modelKK} for the problem of stochastic deconvolution over Lie groups as described in \cite{MR2446754}.

Then, thanks to the Theorem \ref{th:lowerbound} below, there  exists a connection between mean pattern estimation in the setting of Grenander's pattern theory   \cite{Gre,gremil} and the analysis of deconvolution problems in nonparametric statistics. Indeed, in the following theorem, we derive an asymptotic lower bound on $H_s(G,A)$ for the minimax risk $\R(A,s)$ which shows that the rate of convergence $n^{-\frac{2s}{2s+2\nu+\dim(G)}}$ cannot be improved. Thus, $\hat{f}_{T_{n}}$ is an optimal estimator of $f^{\star}$ in the minimax sense.

\begin{thm}\label{th:lowerbound}
Suppose that Assumption \ref{hregu} holds. Let $s > 2  \nu +  \dim G$. Then, there exists a constant $K_{2}>0$ such that
$$\liminf_{n\to\infty} \inf_{\hat{f} \in\LL^2(G)}  \sup_{f^{\star}\in H_s(G,A)}n^{\frac{2s}{2s+2\nu+\dim G}}R(\hat{f},f^{\star})\geq K_{2}.$$
\end{thm}

\begin{appendix}

\section{Technical Appendix}

\subsection{Proof of Theorem \ref{th:upperbound}}


By the classical bias/variance decomposition of the risk one has
$$R(\hat{f}_T^{\star},f^{\star})=\EE\left(\norme{\hat{f}_T^{\star}-\EE\left(\hat{f}_T^{\star}\right)}^2\right)+\norme{\EE\left(\hat{f}_T^{\star}\right) -f^{\star}}^2.$$
Let us first give an upper bound for the bias $\norme{\EE\left(\hat{f}_T^{\star}\right)-f^{\star}}^2$. By linearity of the trace operator and by inverting expectation and sum (since $\Card(\widehat{G}_T)$ is finite) one obtains that
\begin{eqnarray*}
\norme{\EE\left(\hat{f}_T^{\star}\right)-f^{\star}}^2 &=& \norme{\sum_{\pi\in\widehat{G}_T}d_{\pi}\Tr\left[\pi(g)\left(\frac{1}{n}\sum_{m=1}^n\EE\left(c_{\pi}(Y_m)\right)c_{\pi}(h)^{-1}-c_{\pi}(f^{\star})\right)\right] \right. \\
& & \left. -\sum_{\pi\in\widehat{G}\setminus\widehat{G}_T}d_{\pi}\Tr\left[\pi(g)c_{\pi}(f^{\star})\right]}^2.
\end{eqnarray*}
Since the $(c_{\pi}(Y_m))_m$'s are i.i.d.\ random variables and  $\EE(c_{\pi}(Y_m))=\EE(c_{\pi}(f_m))=c_{\pi}(f^{\star})c_{\pi}(h)$ we obtain that$
\norme{\EE\left(\hat{f}_T^{\star}-f^{\star}\right)}^2 =\norme{\sum_{\pi\in\widehat{G}\setminus\widehat{G}_T}d_{\pi}\Tr\left[\pi(g)c_{\pi}(f^{\star})\right]}^2.
$
Then, by Theorem \ref{persevalie}, one has that
$\norme{\EE\left(\hat{f}_T^{\star}-f^{\star}\right)}^2 = \sum_{\pi\in\widehat{G}\setminus\widehat{G}_T}d_{\pi}\Tr\left[c_{\pi}(f^{\star})\overline{c_{\pi}(f^{\star})}^t\right].$
Finally since $\pi\notin\widehat{G}_T$ and $\norme{f}_{H_s}^2\leq A^2$ we obtain the following upper bound for the bias
\begin{equation}
\norme{\EE\left(\hat{f}_T^{\star}-f^{\star}\right)}^2 \leq T^{-s}A^2. \label{eq:boundbias}
\end{equation}
Let us now compute an upper bound for the variance term $\EE\left(\norme{\hat{f}_T^{\star}-\EE\left(\hat{f}_T^{\star}\right)}^2\right)$.
\begin{eqnarray*}
\EE\left(\norme{\hat{f}_T^{\star}-\EE\left(\hat{f}_T^{\star}\right)}^2\right) &=& \EE\left(\norme{\sum_{\pi\in\widehat{G}_T}d_{\pi}\Tr\left[\pi(g)\left(\frac{1}{n}\sum_{m=1}^nc_{\pi}(Y_m)c_{\pi}(h)^{-1}-c_{\pi}(f^{\star})\right)\right]}^2\right)\\
& \leq & 2\underbrace{\EE\left(\norme{\frac{1}{n}\sum_{m=1}^n\sum_{\pi\in\widehat{G}_T}d_{\pi}\Tr\left[\pi(g)\left(c_{\pi}(f_m)c_{\pi}(h)^{-1}-c_{\pi}(f^{\star})\right)\right]}^2\right)}_{E_1}\\
& + & 2\underbrace{\EE\left(\norme{\varepsilon\sum_{\pi\in\widehat{G}_T}d_{\pi}\Tr\left[\pi(g)\frac{1}{n}\sum_{m=1}^nc_{\pi}(W_m)c_{\pi}(h)^{-1}\right]}^2\right)}_{E_2},
\end{eqnarray*}
using that $c_{\pi}(Y_m)=c_{\pi}(f_m)+\varepsilon c_{\pi}(W_m)$.

Let us first consider the term $E_2$. By Theorem \ref{persevalie} and by decomposing the trace
\begin{eqnarray*}
E_2 &=& \varepsilon^2\EE\left(\sum_{\pi\in\widehat{G}_T}d_{\pi}\Tr\left[\frac{1}{n^2}\sum_{m,m'=1}^nc_{\pi}(W_m)c_{\pi}(h)^{-1}\overline{c_{\pi}(W_m')c_{\pi}(h)^{-1}}^t\right]\right)\\
&=& \varepsilon^2\EE\left(\sum_{\pi\in\widehat{G}_T}d_{\pi}\frac{1}{n^2}\sum_{m,m'=1}^n\sum_{k,j=1}^{d_{\pi}}\left(c_{\pi}(W_m)c_{\pi}(h)^{-1}\right)_{kj}\left(\overline{c_{\pi}(W_m')c_{\pi}(h)^{-1}}\right)_{kj}\right)\\
&=& \varepsilon^2\EE\left(\sum_{\pi\in\widehat{G}_T}d_{\pi}\frac{1}{n^2}\sum_{m,m'=1}^n\sum_{k,j=1}^{d_{\pi}}\sum_{i,i'=1}^{d_{\pi}}(c_{\pi}(W_m))_{ki}(c_{\pi}(h)^{-1})_{ij}\overline{(c_{\pi}(W_m')_{k,i'}}\,\overline{((c_{\pi}(h)^{-1})_{i'j}}\right).\\
\end{eqnarray*}
By the Fubini-Tonneli theorem, we can invert sum and integral, and since $\left(((c_{\pi}(W_m))_{kl}\right)_{k,l}$ are i.i.d.\ Gaussian variables with zero expectation and variance $d_{\pi}^{-1}$, it follows that
\begin{eqnarray*}
E_2 &=& \varepsilon^2\sum_{\pi\in\widehat{G}_T}d_{\pi}\frac{1}{n^2}\sum_{m=1}^n\sum_{k,j,i=1}^{d_{\pi}}(c_{\pi}(h)^{-1})_{ij}\overline{(c_{\pi}(h)^{-1})_{ij}}\EE\left((c_{\pi}(W_m))_{ki}(\overline{(c_{\pi}(W_m))_{k,i}}\right)\\
&=& \frac{\varepsilon^2}{n}\sum_{\pi\in\widehat{G}_T}d_{\pi}\sum_{k,j,i=1}^{d_{\pi}}|(c_{\pi}(h)^{-1})_{ij}|^2d_{\pi}^{-1} = \frac{\varepsilon^2}{n}\sum_{\pi\in\widehat{G}_T}d_{\pi}\sum_{j,i=1}^{d_{\pi}}|(c_{\pi}(h)^{-1})_{ij}|^2
\end{eqnarray*}
Then thanks to the properties of the operator norm, one has $\sum_{j=1}^{d_{\pi}}|(c_{\pi}(h)^{-1})_{ij}|^2 \leq \norme{c_{\pi}(h)^{-1}}_{op}^2$, and therefore
\begin{equation}
E_2 \leq \frac{\varepsilon^2}{n}\sum_{\pi\in\widehat{G}_T}d_{\pi}^2\norme{c_{\pi}(h)^{-1}}_{op}^2. \label{eq:boundE2}
\end{equation}
Let us now compute an upper bound for $E_1$. Since $c_{\pi}(f_m)=c_{\pi}(f^{\star})\pi(\tau_{m}^{-1})$ and by Theorem \ref{persevalie},
\begin{eqnarray*}
E_1 &=& \EE\left(\norme{\sum_{\pi\in\widehat{G}_T}d_{\pi}\Tr\left[\pi(g)\left(c_{\pi}(f^{\star})\frac{1}{n}\sum_{m=1}^n\pi(\tau_m^{-1})c_{\pi}(h)^{-1}-c_{\pi}(f^{\star})\right)\right]}^2\right)\\
&=& \EE\left(\sum_{\pi\in\widehat{G}_T}d_{\pi}\norme{(c_{\pi}(f^{\star})\frac{1}{n}\sum_{m=1}^n\pi(\tau_m^{-1})c_{\pi}(h)^{-1}-c_{\pi}(f^{\star})}_F^2\right).
\end{eqnarray*}
By Fubini-Tonelli theorem, we invert sum and integral, and since the random variables $\tau_m$ are i.i.d.\
\begin{eqnarray*}
E_1 &=& \frac{1}{n}\sum_{\pi\in\widehat{G}_T}d_{\pi}\EE\left(\norme{c_{\pi}(f^{\star})\pi(\tau_1^{-1})c_{\pi}(h)^{-1}-c_{\pi}(f^{\star})}_F^2\right)\\
&=& \frac{1}{n}\sum_{\pi\in\widehat{G}_T}d_{\pi}\EE\left(\norme{c_{\pi}(f^{\star})\pi(\tau_1^{-1})c_{\pi}(h)^{-1}}_F^2+\norme{c_{\pi}(f^{\star})}_F^2-2\Tr\left[c_{\pi}(f^{\star})\pi(\tau_1^{-1})c_{\pi}(h)^{-1}\overline{c_{\pi}(f^{\star})}^t\right]\right),
\end{eqnarray*}
where the last equality follows by definition of the Frobenius norm. Now remark that, 
\begin{eqnarray*}
\EE\left(\Tr\left[c_{\pi}(f^{\star})\pi(\tau_1^{-1})c_{\pi}(h)^{-1}\overline{c_{\pi}(f^{\star})}^t\right]\right) &=& \Tr\left[c_{\pi}(f^{\star})\EE(\pi(\tau_1^{-1}))c_{\pi}(h)^{-1}\overline{c_{\pi}(f^{\star})}^t\right]\\
&=& \Tr\left[c_{\pi}(f^{\star})c_{\pi}(h)c_{\pi}(h)^{-1}\overline{c_{\pi}(f^{\star})}^t\right]\\
&=& \norme{c_{\pi}(f^{\star})}_F^2,
\end{eqnarray*}
and let us compute $\EE\left(\norme{c_{\pi}(f^{\star})\pi(\tau_1^{-1})c_{\pi}(h)^{-1}}_F^2\right)$. Recall that
$$
\norme{PQ}_F \leq \norme{P}_{F}\norme{Q}_{op}
$$
 for any $P,Q\in\M_{d_{\pi}\times d_{\pi}}(\CC)$  and that the operator norm is a multiplicative norm, which implies that
\begin{eqnarray*}
\EE\left(\norme{c_{\pi}(f^{\star})\pi(\tau_1^{-1})c_{\pi}(h)^{-1}}_F^2\right) 
&=& \EE\left(\norme{c_{\pi}(f^{\star})}_F^2\norme{\pi(\tau_1^{-1})c_{\pi}(h)^{-1}}_{op}^2\right)\\
&=& \norme{c_{\pi}(f^{\star})}_F^2\EE\left(\norme{\pi(\tau_1^{-1})}_{op}^2\right)\norme{c_{\pi}(h)^{-1}}_{op}^2,
\end{eqnarray*}
Since the operator norm is the smallest matrix norm one has that
$\EE\left(\norme{\pi(\tau_1^{-1})}_{op}^2\right) \leq \EE\left(\norme{\pi(\tau_1^{-1})}_F^2\right).$
Now since
$\norme{\pi(\tau_1^{-1})}_F^2=\Tr\left[\pi(\tau_1^{-1})\overline{\pi(\tau_1^{-1})}^t\right]=\Tr\left[\pi(\tau_1^{-1})\pi(\tau_1)\right]=\Tr[\Id_{d_{\pi}}],$
it follows that
$\EE\left(\norme{\pi(h_1^{-1})}_{op}^2\right)\leq d_{\pi},$
and therefore
\begin{equation}
E_1 \leq \frac{1}{n}\sum_{\pi\in\widehat{G}_T}d_{\pi}\norme{c_{\pi}(f^{\star})}_F^2\left(d_{\pi}\norme{c_{\pi}(h)^{-1}}_{op}^2-1\right).  \label{eq:boundE1}
\end{equation}
Thus, combining the bounds \eqref{eq:boundE2} and \eqref{eq:boundE1} 
\begin{eqnarray*}
\EE\left(\norme{\hat{f}_T^{\star}-\EE\left(\hat{f}_T^{\star}\right)}^2\right) & \leq & \frac{2}{n}\sum_{\pi\in\widehat{G}_T}d_{\pi}^2\left(\norme{c_{\pi}(f^{\star})}_F^2\left(\norme{c_{\pi}(h)^{-1}}_{op}^2-\frac{1}{d_{\pi}}\right)+\varepsilon^2 \norme{c_{\pi}(h)^{-1}}_{op}^2\right)  \\
& \leq & \frac{2}{n}\sum_{\pi\in\widehat{G}_T}d_{\pi}^2\norme{c_{\pi}(h)^{-1}}_{op}^2\left(\norme{c_{\pi}(f^{\star})}_F^2+\varepsilon^2 \right). 
\end{eqnarray*}
Since $f^{\star}\in H_s(G,A)$, this implies that $\norme{c_{\pi}(f^{\star})}_{F}^2 \leq M$, for some constant $M$ that is independent of $\pi$ and $f^{\star}$. 
Hence $\norme{c_{\pi}(f^{\star})}_F^2+\varepsilon^2 \leq (M + \varepsilon^2)$. Assumption \ref{hregu} on the smoothness of $h$ thus implies 
\begin{eqnarray}
\EE\left(\norme{\hat{f}_T^{\star}-\EE\left(\hat{f}_T^{\star}\right)}^2\right) & \leq & \frac{2C_{1}(M + \varepsilon^2)}{n} \sum_{\pi\in\widehat{G}_T} d_{\pi}^2\lambda_{\pi}^{\nu} \leq  \frac{2C_{1}(M + \varepsilon^2)}{n}T^{\nu}\sum_{\pi\in\widehat{G}_T}d_{\pi}^2 \nonumber \\
& \leq & \frac{C}{n}T^{\nu + (\dim(G)/2)}, \label{eq:boundvar}
\end{eqnarray}
where the last inequality follows by Proposition \ref{somdpi}, and $C > 0$ is some constant that is independent of $f^{\star}\in H_s(G,A)$. Therefore, combining the bounds \eqref{eq:boundbias} and \eqref{eq:boundvar} it follows that
$R(\hat{f}_T^{\star},f^{\star})  \leq L(T)$ where $ L(T) = T^{-s}A^2 + \frac{C}{n}T^{\nu + (\dim(G)/2)}$ (note that  $L(T)$ does not depend on $f^{\star} \in H_s(G,A)$). Let us now search among the estimators  $(\hat{f}_T^{\star})_T$ the ones which minimize the upper bound of the quadratic risk. It is clear that the function $T \mapsto L(T)$ has a minimum at  $T = \lfloor n^{\frac{2}{2s+2\nu+\dim(G)}}\rfloor$ such that
$
L(\lfloor n^{\frac{2}{2s+2\nu+\dim(G)}}\rfloor)  \leq  A^2n^{\frac{-2s}{2s+2\nu+\dim(G)}} +\frac{C}{n}n^{\frac{2\nu + \dim(G)}{2s+2\nu+\dim(G)}} \leq  C''n^{\frac{-2s}{2s+2\nu+\dim(G)}}.
$
which completes the proof of Theorem \ref{th:upperbound}. \hfill $\square$


\subsection{Proof of Theorem \ref{th:lowerbound}}


 To obtain a lower bound, we use an adaptation of the Assouad’s cube technique (see e.g.\ \cite{MR2724359} and references therein) to model \eqref{modele4} which differs from the standard white noise models classically studied in nonparametric statistics. Note that for any subset $\Omega \subset H_s(G,A)$  
$$\inf_{\hat{f} }\sup_{f^{\star}\in H_s(G,A)}R(\hat{f},f^{\star})\geq \inf_{\hat{f}  }\sup_{f^{\star}\in \Omega} R(\hat{f},f^{\star}).$$
The main idea is to find an appropriate subset $\Omega$ of test functions that will allow us to compute an asymptotic lower bound for $ \inf_{\hat{f} }\sup_{f^{\star}\in \Omega} R(\hat{f}^{\star},f^{\star})$ and thus the result of Theorem \ref{th:lowerbound} will immediately follow by the above inequality.

\subsubsection{Choice of  a subset $\Omega$ of test functions}

Let us consider a set $\Omega$  of the following form:
$$\Omega = \Omega_D=\left\{f_w^* : G \to \CC :  \forall g\in G, f_w^*(g)=\sqrt{\mu_D} \sum_{\pi\in \widehat{G}_D} d_{\pi} \sum_{k,l}^{d_{\pi}}w_{\pi,kl}(\pi(g))_{kl},\ w_{\pi,kl}\in\{-d_{\pi}^{-1/2},d_{\pi}^{-1/2}\} \right\},$$
where $\widehat{G}_D=\left\{\pi\in\widehat{G}\ :\ D\leq\lambda_{\pi}< 2D\right\}$ and $\mu_D\in\RR_+$. To simplify the presentation of the proof, we will write $f_{w} = f_w^*$. Let $\tilde{\Omega} = \prod_{\pi \in \widehat{G}_D} \{-d_{\pi}^{-1/2},d_{\pi}^{-1/2}\}^{ d_{\pi}^{2}}$. In what follows, the notation $w = (w_{\pi,kl})_{\pi \in \widehat{G}_D, 1 \leq k,l  \leq d_{\pi}} \in \tilde{\Omega}$ is used to denote the set of coefficients $w_{\pi,kl}$ taking their value in $\{-d_{\pi}^{-1/2},d_{\pi}^{-1/2}\}$. The notation $\EE_w$ will be used to denote   expectation  with respect to the distribution $\PP_w$ of the random processes $Y_{m},m \in[\![1,n]\!]$ in model \eqref{modele4} under the hypothesis that $f^{\star} = f_w$.

Note that any $f_w\in\Omega$ can be written as $f_w(g)=\sqrt{\mu_D} \sum_{\pi\in \widehat{G}_D} d_{\pi} \Tr\left[\pi(g)w_{\pi}\right]$, where $w_{\pi}=(w_{\pi,kl})_{1 \leq k,l  \leq d_{\pi}}$. Let $|\Omega|=\Card(\Omega)$ and let us search for a condition on  $\mu_D$ such that  $\Omega\subset H_s(G,A)$. Note that $c_{\pi}(f_w) = \sqrt{\mu_{D}} w_{\pi}$ which implies
\begin{eqnarray*}
f_w\in H_s(G,A) & \iff & \norme{f_w}_{H_s}^2\leq A^2\\
& \iff & \sum_{\pi\in\widehat{G}_D} d_{\pi}\Tr\left[\sqrt{\mu_D}w_{\pi}\overline{\sqrt{\mu_D}w_{\pi}}^t\right] +\sum_{\pi\in\widehat{G}_D}\lambda_{\pi}^sd_{\pi}\Tr\left[\sqrt{\mu_D}w_{\pi}\overline{\sqrt{\mu_D}w_{\pi}}^t\right] \leq A^2\\
& \iff & \sum_{\pi\in\widehat{G}_{D}}(1+\lambda_{\pi}^s) \mu_Dd_{\pi}^2 \leq A^2,
\end{eqnarray*}
using the equality  $\Tr\left[w_{\pi}\overline{w_{\pi}}^t\right]=\sum_{k,l=1}^{d_{\pi}}w_{\pi,kl}^2 = d_{\pi}$ which follows from the fact that $|w_{\pi,kl}| = d_{\pi}^{-1/2}$. Since $\pi\in\widehat{G}_D$, one has that $\lambda_{\pi}< 2D$, and thus
$\mu_D\sum_{\pi\in\widehat{G}_{D}}d_{\pi}^2 \leq 2^{-s}D^{-s}A^2/2 \Longrightarrow \sum_{\pi\in\widehat{G}_{D}}(1+\lambda_{\pi}^s)\mu_Dd_{\pi}^2 \leq A^2$. Moreover by Proposition \ref{somdpi} we have that for $D$ sufficiently large, $\sum_{\pi\in\widehat{G}_{D}}d_{\pi}^2\leq C D^{\dim G/2}$, for some constant $C>0$, and therefore for such a $D$, it follows that
$\mu_D \leq 2^{-s} D^{-s-\dim G/2}(A^2/2)C^{-1} \Longrightarrow \mu_D\sum_{\pi\in\widehat{G}_{D}}d_{\pi}^2 \leq 2^{-s}  D^{-s} A^2/2.$ Hence, there exists a sufficiently large $D_{0}$ such that for all $D \geq D_{0}$ the condition $\mu_D \leq K D^{-s-\dim G/2}$ for some $K > 0$ (independent of $D$) implies that $\Omega\subset H_s(G,A)$. In what follows, we thus assume that 
$ \mu_D = \kappa D^{-s-\dim G/2}$
for some $0 \leq \kappa \leq K$ and $D \geq D_{0}$.

\subsubsection{Minoration of the quadratic risk over $\Omega$}\label{minorationrisqueparti}

Note that the supremum over $\Omega$ of the quadratic risk of any estimator $\hat{f}$ can be bounded from below as follows. First, remark that by Theorem \ref{persevalie}
\begin{eqnarray}
\sup_{f_w\in\Omega}R(\hat{f},f_w)   & = & \sup_{f_w\in\Omega}  \EE_w\left(\norme{\hat{f}-f_w}^2 \right) \nonumber \\
&  \geq & \sup_{w \in \tilde\Omega }  \sum_{\pi\in\widehat{G}_D}d_{\pi}\sum_{k,l=1}^{d_{\pi}}\EE_w\left(\left|(c_{\pi}(\hat{f}))_{kl}-\sqrt{\mu_D}w_{\pi,kl}\right|^2\right) \nonumber \\
& \geq & \frac{1}{|\tilde\Omega|}    \sum_{w \in \tilde\Omega}   \sum_{\pi\in\widehat{G}_D}d_{\pi}     \sum_{k,l=1}^{d_{\pi}}\EE_w\left(\left|(c_{\pi}(\hat{f}))_{kl}-\sqrt{\mu_D}w_{\pi,kl}\right|^2\right) \nonumber \\
& = &  \frac{1}{|\tilde\Omega|}       \sum_{\pi\in\widehat{G}_D}d_{\pi}     \sum_{k,l=1}^{d_{\pi}} \sum_{w \in \tilde\Omega} \EE_w\left(\left|(c_{\pi}(\hat{f}))_{kl}-\sqrt{\mu_D}w_{\pi,kl}\right|^2\right)   \label{eq:sup}
\end{eqnarray}
with $|\Omega| = 2^{\sum_{\pi\in\widehat{G}_D}  d_{\pi}^{2}}$.
Now, define for all $\pi\in\widehat{G}_D,\ k,l\in[\![1,d_{\pi}]\!]$ the coefficients
$$
w_{\pi,kl}^*=\argm_{v\in\left\{-d_{\pi}^{-1/2},d_{\pi}^{-1/2}\right\}}\left|(c_{\pi}(\hat{f}))_{kl}-\sqrt{\mu_D}v\right|.
$$
The inequalities
\begin{eqnarray*}
\sqrt{\mu_D}\left|w_{\pi,kl}-w_{\pi,kl}^*\right| & \leq & \left|\sqrt{\mu_D}w_{\pi,kl}-(c_{\pi}(\hat{f}))_{kl}\right|+\left|\sqrt{\mu_D}w_{\pi,kl}^*-(c_{\pi}(\hat{f}))_{kl}\right|\\
& \leq & 2\left|\sqrt{\mu_D}w_{\pi,kl}-(c_{\pi}(\hat{f}))_{kl}\right|,
\end{eqnarray*}
imply that
$\frac{1}{4}\mu_D\left|w_{\pi,kl}-w_{\pi,kl}^*\right|^2\leq \left|\sqrt{\mu_D}w_{\pi,kl}-(c_{\pi}(\hat{f}))_{kl}\right|^2,$
and thus by inequality \eqref{eq:sup}
\begin{eqnarray}
\sup_{f_w\in\Omega}R(\hat{f},f) & \geq & \frac{\mu_D}{4|\tilde\Omega|}       \sum_{\pi\in\widehat{G}_D}d_{\pi}     \sum_{k,l=1}^{d_{\pi}} \sum_{w \in \tilde\Omega} \EE_w\left(\left|w_{\pi,kl}-w_{\pi,kl}^*\right|^2\right) \nonumber \\
& \geq &  \frac{\mu_D}{4|\tilde\Omega|}       \sum_{\pi\in\widehat{G}_D}d_{\pi}     \sum_{k,l=1}^{d_{\pi}} \sum_{\substack{ w \in \tilde\Omega \\w_{\pi,kl}=d_{\pi}^{-1/2}} }  \EE_w\left(\left|w_{\pi,kl}-w_{\pi,kl}^*\right|^2\right) \nonumber \\
& & +\EE_{w^{(\pi,kl)}}\left(\left|w_{\pi,kl}^{(\pi,kl)}-w_{\pi,kl}^*\right|^2\right),  \label{minrisque1} 
\end{eqnarray}
where for all $\pi'\in\widehat{G}_D,\ k',l'\in[\![1,d_{\pi}]\!]$, we define
$$w^{(\pi,kl)} = (w_{\pi',k'l'}^{(\pi,kl)}) \text{ is such that }\left\{\begin{array}{ll}
w_{\pi',k'l'}^{(\pi,kl)}= w_{\pi',k'l'} & \text{if } \pi'\neq \pi\text{ or }(k',l') \neq (k,l) \\
\\
w_{\pi',k'l'}^{(\pi,kl)}=-  w_{\pi,kl}  & \text{if } \pi'=\pi \text{ and } (k',l') = (k,l) \\
\end{array} \right..$$
 Note that the above minoration depends on $\hat{f}$. Let us introduce the notation $$C_{\pi,kl} :=\EE_w\left(\left|w_{\pi,kl}-w_{\pi,kl}^*\right|^2\right)+\EE_{w^{(\pi,kl)}}\left(\left|w_{\pi,kl}^{(\pi,kl)}-w_{\pi,kl}^*\right|^2\right).$$
In what follows, we show that $C_{\pi,kl}$ can be bounded from below independently  of $\hat{f}$.

\subsubsection{A lower bound for $C_{\pi,kl}$}

Let $\pi\in\widehat{G}_D$, $k,l\in[\![1,d_{\pi}]\!]$ be fixed. Denote by $X=(c_{\pi}(Y_m))_{(\pi,m)\in\widehat{G}\times[\![1,n]\!]}$ the data set in the Fourier domain. In what follows, the notation $\EE_{w,\tau}$ is used to denote expectation with respect to the distribution $\PP_{w,\tau}$ of the random processes $Y_{m},m \in[\![1,n]\!]$ in model \eqref{modele4}  conditionally to $\tau = (\tau_1,\ldots,\tau_n)$ and under the hypothesis that $f^{\star} = f_w$. The notation $w=0$ is used to denote the hypothesis $f^{\star} = 0$ in model \eqref{modele4}. Therefore, using these notations, one can write that
\begin{eqnarray*}
C_{\pi,kl} &=& \int_{G^n} \left[\EE_{w,\tau}\left(\left|w_{\pi,kl}-w_{\pi,kl}^*\right|^2\right)+\EE_{w,\tau}\left(\left|w_{\pi,kl}^{(\pi,kl)}-w_{\pi,kl}^*\right|^2\right)\right]h(\tau_1)...h(\tau_n)\ud \tau_1...\ud\tau_n, \\
&=&\int_{G^n}\EE_{0,\tau}\left(\left|w_{\pi,kl}-w_{\pi,kl}^*\right|^2\frac{\ud\PP_{w,\tau}}{\ud \PP_{0,\tau}}(X)+\left|w_{\pi,kl}^{(\pi,kl)}-w_{\pi,kl}^*\right|^2\frac{\ud\PP_{w^{(\pi,kl)},\tau}}{\ud \PP_{0,\tau}}(X)\right)h(\tau_1)...h(\tau_n)\ud \tau_1...\ud\tau_n\\
&=& \int_{G^n}\EE_{0}\left(\left|w_{\pi,kl}-w_{\pi,kl}^*\right|^2\frac{\ud\PP_{w,\tau}}{\ud \PP_{0}}(X)+\left|w_{\pi,kl}^{(\pi,kl)}-w_{\pi,kl}^*\right|^2\frac{\ud\PP_{w^{(\pi,kl)},\tau}}{\ud \PP_{0}}(X)\right)h(\tau_1)...h(\tau_n)\ud \tau_1...\ud\tau_n,
\end{eqnarray*}
where the last equality follows from the fact that, under the hypothesis $f^{\star} = 0$, the data $X$ in model \eqref{modele4} do not depend on $\tau$. By inverting sum and integral, and using Fubini-Tonneli theorem we obtain 
\begin{eqnarray*}
C_{\pi,kl} &=& \EE_0\left(\int_{G^n}\left(\left|w_{\pi,kl}-w_{\pi,kl}^*\right|^2\frac{\ud\PP_{w,\tau}}{\ud \PP_{0}}(X)+\left|w_{\pi,kl}^{(\pi,kl)}-w_{\pi,kl}^*\right|^2\frac{\ud\PP_{w^{(\pi,kl)},\tau}}{\ud \PP_{0}}(X)\right)h(\tau_1)...h(\tau_n)\ud \tau_1...\ud\tau_n\right)\\
&=& \EE_0\left(\left|w_{\pi,kl}-w_{\pi,kl}^*\right|^2\int_{G^n}\frac{\ud\PP_{w,\tau}}{\ud \PP_{0}}(X)h(\tau_1)...h(\tau_n)\ud \tau_1...\ud\tau_n\right.\\
&+& \left.\left|w_{\pi,kl}^{(\pi,kl)}-w_{\pi,kl}^*\right|^2\int_{G^n}\frac{\ud\PP_{w^{(\pi,kl)},\tau}}{\ud \PP_{0}}(X)h(\tau_1)...h(\tau_n)\ud \tau_1...\ud\tau_n\right).
\end{eqnarray*}
Introduce the notations
$$
Q(X)=\int_{G^n}\frac{\ud\PP_{w,\alpha}}{\ud \PP_{0}}(X)h(\alpha_1)...h(\alpha_n)\ud \alpha_1...\ud \alpha_n \mbox{ and } 
Q^{(\pi,kl)}(X)=\int_{G^n}\frac{\ud\PP_{w^{(\pi,kl)},\alpha}}{\ud \PP_{0}}(X)h(\alpha_1)...h(\alpha_n)\ud \alpha_1...\ud \alpha_n.
$$
Since $w_{\pi,kl}^{(\pi,kl)}-w_{\pi,kl}^* = -w_{\pi,kl}-w_{\pi,kl}^*$ with $w_{\pi,kl} \in \left\{ -d_{\pi}^{-1/2}, d_{\pi}^{-1/2} \right\}  $ and  $w_{\pi,kl}^* \in \left\{ -d_{\pi}^{-1/2}, d_{\pi}^{-1/2} \right\}  $, it follows that
\begin{eqnarray}\label{inegalite3}
C_{\pi,kl} & \geq & 4 d_{\pi}^{-1}\EE_0\left(\min\left(Q(X),Q^{(\pi,kl)}(X)\right)\right)\nonumber\\
& = & 4 d_{\pi}^{-1}\EE_0\left(Q(X)\min\left(1,\frac{Q^{(\pi,kl)}(X)}{Q(X)}\right)\right)\nonumber\\
& = & 4 d_{\pi}^{-1}\EE_0\left(\int_{G^n}\frac{\ud\PP_{w,\tau}}{\ud \PP_{0}}(X)h(\tau_1)...h(\tau_n)\ud \tau_1...\ud\tau_n\min\left(1,\frac{Q^{(\pi,kl)}(X)}{Q(X)}\right)\right)\nonumber\\
& = & 4 d_{\pi}^{-1}\int_{G^n}\EE_0\left(\frac{\ud\PP_{w,\tau}}{\ud \PP_{0}}(X)h(\tau_1)...h(\tau_n)\ud \tau_1...\ud\tau_n\min\left(1,\frac{Q^{(\pi,kl)}(X)}{Q(X)}\right)\right)\nonumber\\
& = & 4 d_{\pi}^{-1}\int_{G^n}\EE_{w,\tau}\left(\min\left(1,\frac{Q^{(\pi,kl)}(X)}{Q(X)}\right)\right)h(\tau_1)...h(\tau_n)\ud \tau_1...\ud\tau_n \nonumber\\
& = & 4 d_{\pi}^{-1} \EE_{w}\left(\min\left(1,\frac{Q^{(\pi,kl)}(X)}{Q(X)}\right)\right).
\end{eqnarray}
Let us now compute a lower bound for $\EE_{w}\left(\min\left(1,\frac{Q^{(\pi,kl)}(X)}{Q(X)}\right)\right)$. Note that for any $0 < \delta < 1$, 
\begin{equation} \label{eq:Markov}
 \EE_{w}\left(\min\left(1,\frac{Q^{(\pi,kl)}(X)}{Q(X)}\right)\right)  \geq \delta\PP_{w}\left(\frac{Q^{(\pi,kl)}(X)}{Q(X)}>\delta\right).
\end{equation}

\begin{prop}\label{prop:final}
Let $\pi\in\widehat{G}_D$, $k,l\in[\![1,d_{\pi}]\!]$ be fixed. Let $\mu_D = \kappa D^{-s-\dim G/2}$ and $D=\left\lfloor n^{\frac{2}{2s+2\nu+\dim G}}\right\rfloor$. Suppose that $s > 2  \nu +  \dim G$. Then, there exists $0 < \delta < 1$ and a constant $C>0$ such that
$$\liminf_{n\to\infty}\PP_{w}\left(\frac{Q^{(\pi,kl)}(X)}{Q(X)}>\delta\right)>C.$$
\end{prop}

\begin{proof}
Throughout the proof, we assume that $\mu_D = \kappa D^{-s-\dim G/2}$ and $D=\left\lfloor n^{\frac{2}{2s+2\nu+\dim G}}\right\rfloor$. To simplify the presentation, we also write $\EE = \EE_{w}$ and  $\PP = \PP_{w}$. Then, thanks to Proposition \ref{somdpi}, it follows that $d_{\pi}^2 \sim D^{(\dim G)/2}$  for $\lambda_{\pi} \in \widehat{G}_{D}$, and therefore, under the assumption that $s > 2  \nu +  \dim G$, one obtains the following relations (needed later on in the proof)
\begin{equation} \label{eq:asympsetting}
n \mu_{D}^{3/2} d_{\pi}^{3}  \to 0 , \; n \mu_{D}^{2} d_{\pi}^{4}  \to 0,  \;  n d_{\pi}^4   \mu_{D}^2 D^{-\nu} \to 0 \;   \mbox{ as }   n \to +\infty,
\end{equation}
and
\begin{equation} \label{eq:mainbigO}
n \mu_{D}  D^{-\nu} = \bigO{1} \;   \mbox{ as }   n \to +\infty.
\end{equation}
Without loss of generality,  we consider the case where $w_{\pi,kl} =  -d_{\pi}^{-1/2}$ and $w_{\pi,kl}^{(\pi,kl)} = d_{\pi}^{-1/2}$ and $\epsilon = 1$.  To simplify the presentation, we also introduce the notation $\tilde{w}_{\pi} = w_{\pi}^{(\pi,kl)}$. In the proof, we also make repeated use of the fact that 
\begin{equation}
\| w_{\pi} \|^{2}_{F} = d_{\pi} \mbox{ and } \| \tilde{w}_{\pi} \|^{2}_{F} = d_{\pi}. \label{eq:boundnormw}
\end{equation}
Since $c_{\pi} (Y_{m}) = \sqrt{\mu_{D}} w_{\pi} \pi(\tau_{m}^{-1})  +  c_{\pi}(W_{m}) $ (under the hypothesis that $f^{\star} = f_w$) and using the fact that $   \| \tilde{w}_{\pi} \|^{2}_{F} =    \| w_{\pi}  \|^{2}_{F}$, simple calculations on the likelihood ratios $\frac{\ud\PP_{w,\alpha}}{\ud \PP_{0}}(X)$ and $\frac{\ud\PP_{w^{(\pi,kl)},\alpha}}{\ud \PP_{0}}(X)$ yield that
$$
 \frac{Q^{(\pi,kl)}(X)}{Q(X)} = \frac{\prod_{m=1}^{n}    \int_{G} \exp (\tilde{Z}_{m}^{(1)} + \tilde{Z}_{m}^{(2)} ) h(\alpha_{m}) \ud  \alpha_{m}  }{\prod_{m=1}^{n}    \int_{G} \exp (Z_{m}^{(1)} + Z_{m}^{(2)}  ) h(\alpha_{m}) \ud  \alpha_{m}  }
$$
where
\begin{eqnarray*}
\tilde{Z}_{m}^{(1)}  = d_{\pi}   \mu_{D}  \langle w_{\pi} \pi(\tau_{m}^{-1}) , \tilde{w}_{\pi} \pi(\alpha_{m}^{-1}) \rangle_{F},  & \quad & \tilde{Z}_{m}^{(2)}  =  d_{\pi} \sqrt{\mu_{D}}  \langle c_{\pi} (W_{m}) , \tilde{w}_{\pi}  \pi(\alpha_{m}^{-1}) \rangle_{F}, \\
Z_{m}^{(1)}  = d_{\pi}   \mu_{D}  \langle w_{\pi} \pi(\tau_{m}^{-1}) , w_{\pi}  \pi(\alpha_{m}^{-1}) \rangle_{F},  & \quad & Z_{m}^{(2)}  =  d_{\pi} \sqrt{\mu_{D}}  \langle c_{\pi} (W_{m}) , w_{\pi}  \pi(\alpha_{m}^{-1}) \rangle_{F}.
\end{eqnarray*}
Note that by Cauchy-Schwarz's inequality 
$$
|\tilde{Z}_{m}^{(1)}|^{2} \leq d_{\pi}^{2}   \mu_{D}^{2} \| w_{\pi} \pi(\tau_{m}^{-1}) \|^{2}_{F}  \| \tilde{w}_{\pi} \pi(\alpha_{m}^{-1}) \|^{2}_{F} = d_{\pi}^{2}   \mu_{D}^{2} \| w_{\pi}   \|^{2}_{F}  \| \tilde{w}_{\pi}   \|^{2}_{F},
$$
and
$$
|Z_{m}^{(1)}|^{2} \leq d_{\pi}^{2}   \mu_{D}^{2} \| w_{\pi} \pi(\tau_{m}^{-1}) \|^{2}_{F}  \| w_{\pi} \pi(\alpha_{m}^{-1}) \|^{2}_{F} = d_{\pi}^{2}   \mu_{D}^{2} \| w_{\pi}   \|^{4}_{F} . 
$$
Since the coefficients  of the matrix $c_{\pi}(W_m)$ are independent complex Gaussian random variables   with zero expectation and variance $d_{\pi}^{-1}$, one has  that $\tilde{Z}_{m}^{(2)}$ (resp.\ $Z_{m}^{(2)}$) is a Gaussian random variable with zero mean and variance $d_{\pi} \mu_{D} \| \tilde{w}_{\pi}  \pi(\alpha_{m}^{-1}) \|^{2}_{F}  = d_{\pi} \mu_{D} \| \tilde{w}_{\pi}  \|^{2}_{F}$ (resp.\ $d_{\pi} \mu_{D} \| w_{\pi}  \pi(\alpha_{m}^{-1}) \|^{2}_{F}  = d_{\pi} \mu_{D} \| w_{\pi}  \|^{2}_{F} $). Thence, by  \eqref{eq:boundnormw}, one obtains that 
\begin{equation} \label{eq:moments}
\EE | \tilde{Z}_{m}^{(1)} |^2  \leq  \mu_{D}^{2} d_{\pi}^4, \; \EE | Z_{m}^{(1)} |^2  \leq  \mu_{D}^{2} d_{\pi}^4  \mbox{ and }   \EE | \tilde{Z}_{m}^{(2)} |^2  =  \mu_{D} d_{\pi}^2, \;  \EE | Z_{m}^{(2)} |^2  =    \mu_{D} d_{\pi}^2. 
\end{equation}
Therefore, \eqref{eq:asympsetting} and Markov's inequality imply that
\begin{equation} \label{eq:tildeZ}
|\tilde{Z}_{m}^{(1)}|^2 = \smallOp{n^{-1}}, \; |\tilde{Z}_{m}^{(2)}|^3 = \smallOp{n^{-1}},  \;  |\tilde{Z}_{m}^{(1)} \tilde{Z}_{m}^{(2)}|  = \smallOp{n^{-1}},
\end{equation}
and
\begin{equation} \label{eq:Z}
|Z_{m}^{(1)}|^2 = \smallOp{n^{-1}}, \; |Z_{m}^{(2)}|^3 = \smallOp{n^{-1}}, \; |Z_{m}^{(1)} Z_{m}^{(2)}|  = \smallOp{n^{-1}}.
\end{equation}
Hence, using    \eqref{eq:tildeZ}, \eqref{eq:Z} and  the second order Taylor expansion $\exp(z) = 1+z + \frac{z^2}{2} + \bigO{z^{3}}$ it follows that
\begin{eqnarray*}
\log \left(  \frac{Q^{(\pi,kl)}(X)}{Q(X)} \right) &=& \sum_{m=1}^{n} \log \left( 1 + \int_{G} \left( \tilde{Z}_{m}^{(1)} + \tilde{Z}_{m}^{(2)} + \frac{1}{2} |\tilde{Z}_{m}^{(2)}|^2  \right) h(\alpha_{m}) \ud \alpha_{m} + \smallOp{n^{-1}} \right) \\
& & - \sum_{m=1}^{n} \log \left( 1 + \int_{G} \left( Z_{m}^{(1)} + Z_{m}^{(2)} +  \frac{1}{2} |Z_{m}^{(2)}|^2  \right) h(\alpha_{m}) \ud \alpha_{m} + \smallOp{n^{-1}} \right).
\end{eqnarray*}
Then,  using   \eqref{eq:Z} and the second order expansion $\log(1 + z) = z - \frac{z^2}{2} + \bigO{z^3} $ yield
\begin{eqnarray}
\log \left(  \frac{Q^{(\pi,kl)}(X)}{Q(X)} \right) &=& \sum_{m=1}^{n}  \int_{G} \left( \tilde{Z}_{m}^{(1)} + \tilde{Z}_{m}^{(2)} + \frac{1}{2} |\tilde{Z}_{m}^{(2)}|^2 \right) h(\alpha_{m}) \ud \alpha_{m}  \nonumber \\
& & -\frac{1}{2}  \left[    \int_{G} \left( \tilde{Z}_{m}^{(1)} + \tilde{Z}_{m}^{(2)} + \frac{1}{2} |\tilde{Z}_{m}^{(2)}|^2 \right) h(\alpha_{m}) \ud \alpha_{m} \right]^{2} \label{eq:termtildeZ} \\
& & - \sum_{m=1}^{n}  \int_{G} \left( Z_{m}^{(1)} + Z_{m}^{(2)} +  \frac{1}{2} |Z_{m}^{(2)}|^2 \right) h(\alpha_{m}) \ud \alpha_{m} \nonumber  \\
& & +\frac{1}{2}  \left[    \int_{G} \left( Z_{m}^{(1)} + Z_{m}^{(2)} + \frac{1}{2} |Z_{m}^{(2)}|^2  \right) h(\alpha_{m}) \ud \alpha_{m} \right]^{2} \label{eq:termZ}  \\ 
& & + \smallOp{1} \nonumber . 
\end{eqnarray}
Let us now study the expansion of the quadratic term \eqref{eq:termZ}. Since $c_{\pi}(h) = \int_{G} \pi(\tau_{m}^{-1}) h(\alpha_{m}) \ud \alpha_{m}$, it follows by Cauchy-Schwarz's inequality  that
\begin{eqnarray*}
\sum_{m=1}^{n}  \left[ \int_{G}  Z_{m}^{(1)} h(\alpha_{m}) \ud \alpha_{m}  \right]^{2} & = & d_{\pi}^2   \mu_{D}^2 \sum_{m=1}^{n}   \langle w_{\pi} \pi(\tau_{m}^{-1}) , w_{\pi} c_{\pi}(h) \rangle_{F}^{2} \leq n d_{\pi}^2   \mu_{D}^2 \| w_{\pi} \|^{2}_{F}  \| w_{\pi} c_{\pi}(h)  \|^{2}_{F} \\
& \leq &  n d_{\pi}^2   \mu_{D}^2 \| w_{\pi} \|^{4}_{F}    \| c_{\pi}(h)  \|^{2}_{op} \leq C_{2} n d_{\pi}^4   \mu_{D}^2 D^{-\nu} = \smallO{1}.
\end{eqnarray*}
for some constant $C_{2} > 0$, where the last inequality is a consequence of Assumption \ref{hregu}, the fact that $\lambda_{\pi}^{-1} \leq D^{-1}$ for $\lambda_{\pi} \in \widehat{G}_{D}$ and the third relation in \eqref{eq:asympsetting}.

By Jensen's inequality and \eqref{eq:asympsetting} and since the  $Z_{m}^{(2)}$'s are i.i.d.\   Gaussian random variables with zero mean and variance $\mu_{D} d_{\pi}^{2} $ one obtains that
\begin{eqnarray*}
\EE \sum_{m=1}^{n}  \left[ \int_{G}  |Z_{m}^{(2)}|^{2} h(\alpha_{m}) \ud \alpha_{m}  \right]^{2} & \leq & \sum_{m=1}^{n}  \int_{G} \EE  |Z_{m}^{(2)}|^{4}  h(\alpha_{m}) \ud \alpha_{m} \leq  3 n  \mu_{D}^2  d_{\pi}^{4}= \smallO{1},
\end{eqnarray*}
and thus  Markov's inequality implies that $ \sum_{m=1}^{n}  \left[ \int_{G}  |Z_{m}^{(2)}|^{2} h(\alpha_{m}) \ud \alpha_{m}  \right]^{2} = \smallOp{1}$.
Now, using  \eqref{eq:asympsetting} and  \eqref{eq:moments} it follows that
$$
\EE \left| \sum_{m=1}^{n}  \left( \int_{G}  Z_{m}^{(1)} h(\alpha_{m}) \ud \alpha_{m}  \right) \left( \int_{G}  Z_{m}^{(2)} h(\alpha_{m}) \ud \alpha_{m}  \right) \right| \leq n \mu_{D}^{3/2} d_{\pi}^{3} =\smallO{1},
$$
which implies that $ \sum_{m=1}^{n}  \left( \int_{G}  Z_{m}^{(1)} h(\alpha_{m}) \ud \alpha_{m}  \right) \left( \int_{G}  Z_{m}^{(2)} h(\alpha_{m}) \ud \alpha_{m}  \right) = \smallOp{1}$. Finally, using \eqref{eq:Z}, it follows  that $\sum_{m=1}^{n}  \left( \int_{G}  Z_{m}^{(2)} h(\alpha_{m}) \ud \alpha_{m}  \right) \left( \int_{G}  |Z_{m}^{(2)}|^{2} h(\alpha_{m}) \ud \alpha_{m}  \right) = \smallOp{1}$. By applying the same arguments to the  expansion of the quadratic term \eqref{eq:termtildeZ}, one finally obtains that
\begin{eqnarray*}
\log \left(  \frac{Q^{(\pi,kl)}(X)}{Q(X)} \right) &=& \sum_{m=1}^{n}  \int_{G} \left( \tilde{Z}_{m}^{(1)} + \tilde{Z}_{m}^{(2)} + \frac{1}{2} |\tilde{Z}_{m}^{(2)}|^2 \right) h(\alpha_{m}) \ud \alpha_{m}    \\
& & -\frac{1}{2}  \left[    \int_{G}  \tilde{Z}_{m}^{(2)} h(\alpha_{m}) \ud \alpha_{m} \right]^{2}   \\
& & - \sum_{m=1}^{n}  \int_{G} \left( Z_{m}^{(1)} + Z_{m}^{(2)} +  \frac{1}{2} |Z_{m}^{(2)}|^2 \right) h(\alpha_{m}) \ud \alpha_{m}     \\
& & +\frac{1}{2}  \left[    \int_{G}  Z_{m}^{(2)}  h(\alpha_{m}) \ud \alpha_{m} \right]^{2}   \\ 
& & + \smallOp{1}.
\end{eqnarray*}
Using that $   \| \tilde{w}_{\pi} \|^{2}_{F} =    \| w_{\pi}  \|^{2}_{F}$ and the equality
$$
\langle - w_{\pi}c_{\pi}(h), (\tilde{w}_{\pi} - w_{\pi}) c_{\pi}(h)  \rangle_{F} - \frac{1}{2} \| w_{\pi}c_{\pi}(h) \|^{2}_{F} + \frac{1}{2} \| \tilde{w}_{\pi}c_{\pi}(h) \|^{2}_{F}  -  \frac{1}{2} \| (\tilde{w}_{\pi}-w_{\pi})  c_{\pi}(h) \|^{2}_{F}= 0
$$
one obtains that
\begin{eqnarray}
\log \left(  \frac{Q^{(\pi,kl)}(X)}{Q(X)} \right) &=&  \sum_{m=1}^{n} d_{\pi}    \mu_{D}  \langle w_{\pi} \pi(\tau_{m}^{-1}) - w_{\pi} c_{\pi}(h)  , (\tilde{w}_{\pi} - w_{\pi}) c_{\pi}(h)   \rangle_{F}  \label{eq:A1} \\
& &  +   \sum_{m=1}^{n}  d_{\pi} \sqrt{\mu_{D}}  \langle c_{\pi} (W_{m}) , (\tilde{w}_{\pi} - w_{\pi}) c_{\pi}(h)  \rangle_{F} \label{eq:A2}   \\
& &  +   \sum_{m=1}^{n}  \frac{1}{2}      \int_{G}  | \tilde{Z}_{m}^{(2)}|^{2} h(\alpha_{m}) \ud \alpha_{m}     - \frac{n}{2}   d_{\pi} \mu_{D} \| \tilde{w}_{\pi} \|^{2}_{F}  \label{eq:A3}  \\
& & -  \sum_{m=1}^{n}  \frac{1}{2}     \int_{G}  | Z_{m}^{(2)} |^{2} h(\alpha_{m}) \ud \alpha_{m}  +  \frac{n}{2} d_{\pi} \mu_{D}  \| w_{\pi} \|^{2}_{F}   \label{eq:A4}  \\
& &  -  \sum_{m=1}^{n} \frac{1}{2}  \left(    \int_{G}  \tilde{Z}_{m}^{(2)} h(\alpha_{m}) \ud \alpha_{m} \right)^{2}  + \frac{n}{2}  d_{\pi} \mu_{D}  \| \tilde{w}_{\pi} c_{\pi}(h)\|_{F}^{2}   \label{eq:A5} \\
& & +  \sum_{m=1}^{n}  \frac{1}{2}  \left(    \int_{G}  Z_{m}^{(2)} h(\alpha_{m}) \ud \alpha_{m} \right)^{2}  - \frac{n}{2}    d_{\pi} \mu_{D}  \| w_{\pi} c_{\pi}(h)\|_{F}^{2}   \label{eq:A6} \\
& &  - \frac{n}{2} d_{\pi} \mu_{D}   \|(\tilde{w}_{\pi} - w_{\pi}) c_{\pi}(h)\|_{F}^{2}  + \smallOp{1}.  \label{eq:A7}
\end{eqnarray}

\noindent {\bf Control of the term \eqref{eq:A7}.}  Thanks to  Assumption \ref{hregu} and  the fact that $\lambda_{\pi}^{-1} \leq D^{-1}$ for $\lambda_{\pi} \in \widehat{G}_{D}$ and  \eqref{eq:asympsetting}, it follows by \eqref{eq:mainbigO} that 
\begin{equation}
n d_{\pi} \mu_{D}   \|(\tilde{w}_{\pi} - w_{\pi}) c_{\pi}(h)\|_{F}^{2} \leq n d_{\pi} \mu_{D} \| c_{\pi}(h)\|_{op}^{2}   \|\tilde{w}_{\pi} - w_{\pi} \|_{F}^{2} \leq 4 n   \mu_{D} D^{-\nu} = \bigO{1}, \label{eq:bigO}
\end{equation}
and thus the term \eqref{eq:A7} is bounded in probability. \\

\noindent {\bf Control of the term \eqref{eq:A1}.} Remark that \eqref{eq:asympsetting} can be used to prove that
\begin{eqnarray*}
\Var \left( \sum_{m=1}^{n} d_{\pi}    \mu_{D}  \langle w_{\pi} \pi(\tau_{m}^{-1})   , (\tilde{w}_{\pi} - w_{\pi}) c_{\pi}(h)   \rangle_{F} \right) & \leq & n d_{\pi}^2    \mu_{D}^2 \| w_{\pi} \|_{F}^2   \| (\tilde{w}_{\pi} - w_{\pi}) c_{\pi}(h)    \|_{F}^2 \\
& \leq &  n d_{\pi}^2    \mu_{D}^2 \| w_{\pi} \|_{F}^2   \| \tilde{w}_{\pi} - w_{\pi} \|_{F}^2 \| c_{\pi}(h)    \|_{op}^2 \\
& \leq & 4 n d_{\pi}^2    \mu_{D}^2  D^{-\nu} = \smallO{1},
\end{eqnarray*}
and therefore by Chebyshev's inequality the term \eqref{eq:A1} converges to zero in probability. \\

\noindent {\bf Control of the term \eqref{eq:A2}.}   First, since the coefficients  of the matrix $c_{\pi}(W_m)$ are independent complex Gaussian random variables   with zero expectation and variance $d_{\pi}^{-1}$, one has that $T_{m} = d_{\pi} \sqrt{\mu_{D}}  \langle c_{\pi} (W_{m}) , (\tilde{w}_{\pi} - w_{\pi}) c_{\pi}(h)  \rangle_{F}$ are i.i.d.\ Gaussian random variables with zero mean and variance $ d_{\pi}  \mu_{D}  \|(\tilde{w}_{\pi} - w_{\pi}) c_{\pi}(h)\|_{F}^{2}$. Using inequality \eqref{eq:bigO} it follows that
\begin{equation}
\Var \left(   \sum_{m=1}^{n} T_{m}  \right) = n d_{\pi}  \mu_{D}  \|(\tilde{w}_{\pi} - w_{\pi}) c_{\pi}(h)\|_{F}^{2} = \bigO{1}, \label{eq:boundvar}
\end{equation}
and standard arguments in concentration of Gaussian variables imply that  for any $t > 0$
\begin{equation}
\PP \left(  \left| \sum_{m=1}^{n}  T_{m}  \right| \geq t \right) \leq 2 \exp\left( - \frac{t^{2}}{2  n d_{\pi}  \mu_{D}  \|(\tilde{w}_{\pi} - w_{\pi}) c_{\pi}(h)\|_{F}^{2} }\right). \label{eq:concen}
\end{equation}
Therefore, combining \eqref{eq:boundvar} and \eqref{eq:concen} imply that the term \eqref{eq:A2} is bounded in probability. \\

\noindent {\bf Control of the terms \eqref{eq:A3} and \eqref{eq:A4}.} Remark that Jensen's inequality, the fact that   the $\tilde{Z}_{m}^{(2)}$'s are i.i.d.\   Gaussian random variables with zero mean and variance $\mu_{D} d_{\pi}^{2} $ and \eqref{eq:asympsetting}  imply that
\begin{eqnarray*}
\Var \left(  \sum_{m=1}^{n}      \int_{G}  | \tilde{Z}_{m}^{(2)}|^{2} h(\alpha_{m}) \ud \alpha_{m}  \right) & = &  \sum_{m=1}^{n}    \Var \left(  \int_{G}  | \tilde{Z}_{m}^{(2)}|^{2} h(\alpha_{m}) \ud \alpha_{m} \right)   \\
& \leq & \sum_{m=1}^{n}    \EE \left(  \int_{G}  | \tilde{Z}_{m}^{(2)}|^{2} h(\alpha_{m}) \ud \alpha_{m} \right)^{2} \\ 
& \leq & n   \int_{G}  \EE  | \tilde{Z}_{1}^{(2)}|^{4} h(\alpha_{1}) \ud \alpha_{1} \leq 3 n  \mu_{D}^2  d_{\pi}^{4}= \smallO{1},
\end{eqnarray*}
and thus the terms  \eqref{eq:A3} and \eqref{eq:A4} converge to zero in probability by Chebyshev's inequality. \\

\noindent {\bf Control of the terms \eqref{eq:A5} and \eqref{eq:A6}.} Similarly, by Jensen's inequality and  \eqref{eq:asympsetting} one has that
$$
\Var \left(  \sum_{m=1}^{n}  \left(    \int_{G}  \tilde{Z}_{m}^{(2)} h(\alpha_{m}) \ud \alpha_{m} \right)^{2} \right) \leq n  \int_{G} \EE |  \tilde{Z}_{1}^{(2)} |^{4} h(\alpha_{1}) \ud \alpha_{1}  \leq 3 n  \mu_{D}^2  d_{\pi}^{4}= \smallO{1},
$$
and thus the terms  \eqref{eq:A5} and \eqref{eq:A6} converge to zero in probability by Chebyshev's inequality. \\

Combining the above controls of the terms \eqref{eq:A1} to \eqref{eq:A7}, one obtains that $\log \left(  \frac{Q^{(\pi,kl)}(X)}{Q(X)} \right)$ is bounded in probability   which completes the proof of Proposition \ref{prop:final}.
\end{proof}

Now, recall that using \eqref{minrisque1} and \eqref{inegalite3}
\begin{eqnarray*}
\sup_{f_w\in\Omega}R(\hat{f},f) & \geq &  \frac{\mu_D}{4|\tilde\Omega|}       \sum_{\pi\in\widehat{G}_D}d_{\pi}     \sum_{k,l=1}^{d_{\pi}} \sum_{\substack{ w \in \tilde\Omega \\w_{\pi,kl}=d_{\pi}^{-1/2}} }  C_{\pi,kl}\\
& \geq &\frac{\mu_D}{4|\tilde\Omega|}       \sum_{\pi\in\widehat{G}_D}d_{\pi}     \sum_{k,l=1}^{d_{\pi}} \sum_{\substack{ w \in \tilde\Omega \\w_{\pi,kl}=d_{\pi}^{-1/2}} }   4 d_{\pi}^{-1} \EE_{w}\left(\min\left(1,\frac{Q^{(\pi,kl)}(X)}{Q(X)}\right)\right)
\end{eqnarray*}
Combining inequality \eqref{eq:Markov} and Proposition \ref{prop:final} one obtains that there exists a constant $C > 0$ (not depending on $n$) such that with the choice $D = \left\lfloor n^{\frac{2}{2s+2\nu+\dim G}}\right\rfloor$ and for all sufficiently large $n$  
\begin{eqnarray*}
\sup_{f_w\in\Omega}R(\hat{f},f) & \geq &     \frac{\mu_D}{|\tilde\Omega|}       \sum_{\pi\in\widehat{G}_D}d_{\pi}     \sum_{k,l=1}^{d_{\pi}} \sum_{\substack{ w \in \tilde\Omega \\w_{\pi,kl}=d_{\pi}^{-1/2}} }   d_{\pi}^{-1} C \\
& \geq &\frac{C}{2} \mu_D\sum_{\pi\in\widehat{G}_D}d_{\pi}^2,
\end{eqnarray*}
where we have the fact that for any $\pi,k,l$ the cardinality of the set $\{ w \in \tilde\Omega \mbox{ with } w_{\pi,kl}=d_{\pi}^{-1/2} \}$ is $|\tilde\Omega|/2$.

Now, let $0 < \rho < 1$. Thanks to Proposition \ref{somdpi}, it follows that for $\eta  =  \rho W \frac{2^{\dim G/2}-1}{2^{\dim G/2}+1}$, one has that
$(W+\eta) D^{\dim G/2} \geq \sum_{\pi\in\widehat{G}_D}d_{\pi}^2\geq (W-\eta)  D^{\dim G/2}$  for all sufficiently large $D$, where $W$ is the constant defined in \eqref{eq:W}.
Hence,
\begin{eqnarray*}
\sum_{\pi\in\widehat{G}_D}d_{\pi}^2 &=& \sum_{\pi\,:\,\lambda_{\pi}<2D}d_{\pi}^2-\sum_{\pi\,:\,\lambda_{\pi}<D}d_{\pi}^2\\
&\geq & (W-\eta)(2D)^{\dim G/2}-(W+\eta) D^{\dim G/2}\\
&=& W'D^{\dim G/2}, \mbox{ with } W' = (1-\rho) W (2^{\dim G/2}-1) > 0.
\end{eqnarray*}
Taking $D=\left\lfloor n^{\frac{2}{2s+2\nu+\dim G}}\right\rfloor$ and since $\mu_D = \kappa D^{-s-\dim G/2}$ we finally obtain that 
\begin{eqnarray*}
n^{\frac{2s}{2s+2\nu+\dim G}}\sup_{f_{w}\in\Omega}R(\hat{f},f) & \geq & n^{\frac{2s}{2s+2\nu+\dim G}}KD^{-s-\dim G/2}D^{\dim G/2} =K,
\end{eqnarray*}
for some constant $K > 0$ not depending on $n$, which completes the proof of Theorem \ref{th:lowerbound}.


\section{Some background on noncommutative harmonic analysis} \label{app:Liegroup}

In this appendix, some aspects of the theory of the Fourier transform on compact Lie groups are summarized. For detailed introductions to Lie groups and noncommutative harmonic analysis we refer to the books \cite{MR2062813,MR1738431,MR2279709}. Throughout the Appendix, it is assumed that  $G$ is a connected and compact Lie group.

\subsection{Representations}

\begin{defi}
Let $V$ be a finite-dimensional $\CC$-vector space. A  representation of $G$ in $V$ is a continuous homomorphism
$\pi:G\to GL(V),$
where $GL(V)$ is the set of automorphisms of $V$. The representation $\pi$ is said to be irreducible if, for any $g\in G$, the only invariant subspaces by the automorphism $\pi(g)$ are $\{0\}$ and $V$.
\end{defi}
If  $G$ is a compact group and $\pi$ is an irreducible representation in $V$, then the vector space $V$ is finite dimensional, and we denote by $d_{\pi}$ the dimension of $V$.  By choosing a basis for $V$, it is often convenient to identify $\pi(g)$ with a matrix of size $d_{\pi} \times d_{\pi}$ with complex entries.
\begin{defi}
Two representations $\pi,\,\pi'$ in $V$ are called equivalent if there exists $M\in GL(V)$ such that  $\pi (g)=M\pi'(g)M^{-1}$ for all $g\in G$.
\end{defi}
\begin{defi}
A representation $\pi$ is said to be unitary if $\pi(g)$ is a unitary operator for every $g\in G$.
\end{defi}
Let $\pi$ be a representation in $V$.  Then, there exists an inner product on $V$ such that $\pi$ is unitary. This means that any  irreducible representation $\pi$ in $V$ is  equivalent to an irreducible representation that is unitary.
\begin{defi}
We denote by $\widehat{G}$  the set of equivalence classes of irreducible representations of $G$, and we identify $\widehat{G}$ to the set of unitary representations of each class. 
\end{defi}
\begin{prop}
Let $g\in G$ and $\pi\in\widehat{G}$, then  
$\pi(g^{-1})=\overline{\pi(g)}^t.$
\end{prop}

\subsection{Peter-Weyl theorem}\label{PWpartie}

Let $\pi\in\widehat{G}$ be a representation in a Hilbert space $V$. Let $\B_{\pi}=(e_1,...,e_{d_{\pi}})$ a basis of $V$. For $g\in G$, denote by $\phi_{ij}^{\pi}(g)=\langle e_i,\pi(g)e_j\rangle$ the coordinates of  $\pi$ in the basis $\B_{\pi}$ for $ i,j\in[\![1,d_{\pi}]\!]$.

\begin{thm} \label{theo:PW}
If $G$ is a compact group then $\left(\sqrt{d_{\pi}}\phi_{ij}^{\pi}(.)\right)_{\pi\in\widehat{G},\ i,j\in[\![1,d_{\pi}]\!]}$  is an orthonormal basis of the Hilbert space $\LL^2(G)$ endowed with the inner product $\langle f, h \rangle = \int_G f(g) \overline{h(g)} dg $.
\end{thm}

%

\subsection{Fourier transform and convolution in $\LL^2(G)$} \label{sec:parseval}

Let $\pi\in\widehat{G}$ and define for any $f\in\LL^2(G)$ the linear mapping
\begin{eqnarray*}
c_{\pi}(f):\,V & \to & V\\
v & \mapsto & \int_G f(g)\overline{\pi(g)} ^T v\ud g =\int_G f(g)\pi(g^{-1}) v\ud g.
\end{eqnarray*} 
Note that the matrix $c_{\pi}(f)$ is the generalization to functions in $\LL^2(G)$  of the usual notion of Fourier coefficients.  

\begin{defi}\label{coeffourierlie}
Let  $f\in\LL^2(G)$ and $\pi\in\widehat{G}$. We call $c_{\pi}(f)$ the $\pi$-th Fourier coefficient of $f$.
\end{defi}
\begin{thm}\label{persevalie}
Let $f\in\LL^2 (G)$. Then $f(g)=\sum_{\pi\in \widehat{G}}d_{\pi}\Tr\left(\pi(g)c_{\pi}(f)\right),$
and 
$ ||f||_{\LL^2(G)}^2=\sum_{\pi\in\widehat{G}}d_{\pi}\Tr\left(c_{\pi}(f)\overline{c_{\pi}(f)}^t\right)=\sum_{\pi\in\widehat{G}}d_{\pi}\norme{c_{\pi}(f)}_F^2,$
where $\norme{\cdot}_F$ denotes the Frobenius norm of a matrix.
\end{thm}

\begin{defi}
Let $f,h \in\LL^2(G)$. The convolution of $f$ and $h$ is defined as the function
$(f*h)(g)=\int_G f(g'^{-1}g)h(g')\ud g'$ for $g \in G$.
\end{defi}

\begin{prop}
Let  $f,h\in\LL^2(G)$ then $c_{\pi}(f*h)=c_{\pi}(f)c_{\pi}(h)$.
\end{prop}

\section{Laplace-Beltrami operator on a compact Lie group}

For further details on the material presented in this section we refer to the technical appendix in \cite{MR2446754} and to the book \cite{MR2426516}. In this section, we still assume that $G$ is a connected and compact Lie group. In what follows, with no loss of generality, we  identify (through an isomorphism) $G$ to a subgroup of $GL_{r \times r}(\CC)$ (the set of $r \times r$ nonsingular matrices with complex entries) for some integer $r > 0$.

\subsection{Lie algebra}

\begin{defi}
A one parameter subgroup of $G$ is a group homomorphism  $c:\RR\to G.$
\end{defi}

\begin{thm}
Let $c:\RR\to GL_{r \times r}(\CC)$ one parameter subgroup of $GL_{r \times r}(\CC)$. Then $c$ is $\C^{\infty}$ and
$c(t)=\exp(tA),$
with $A=\dfrac{\ud c}{\ud t}(0)$.
\end{thm}

\begin{defi}
Let $\M_{r \times r}(\CC)$ be the set of $r \times r$ matrices with complex entries. The mapping $[.,.]:\M_{r \times r}(\CC)^2\to \M_{r \times r}(\CC):X,Y\mapsto [X,Y]=XY-YX$ is called a Lie bracket. A Lie algebra is the $\CC$-vector space $\mathfrak{g}=\left\{X\in\M_{r \times r}(\CC): \; \exp(tX)\in G \; \forall t\in\RR\right\}$ endowed with the bilinear form
$[.,.]:\mathfrak{g}\times\mathfrak{g}\to\mathfrak{g}: X,Y\mapsto [X,Y], $
which satisfies  $[X,Y]=-[Y,X]$ and
$\left[[X,Y],Z\right]+\left[[Y,Z],X\right]+\left[[Z,X],Y\right]=0$ (Jacobi identity). 
\end{defi}

\begin{defi}
The Killing form is the bilinear form $B$ defined by 
$$B:\mathfrak{g}\to \CC: X,Y\mapsto\Tr\left[ad(X)ad(Y)\right],$$
where $ad(X):\mathfrak{g}\to\mathfrak{g}:Y\mapsto[X,Y]$ is an endomorphism of $\mathfrak{g}$.
\end{defi}

\subsection{Roots of a Lie algebra}

A torus in $G$ is a connected Abelian subgroup  of $G$. It is well known that in a compact Lie group $G$, there exists (up to an isomorphism) a maximal torus. Let us fix such a maximal torus that we denote by $\TT$. Denote by $\mathfrak{t}$ the Lie algebra  of $\TT$, which is a maximal Abelian subalgebra of $\mathfrak{g}$. Let  $\mathfrak{h}=\mathfrak{t}+i\mathfrak{t}$ be the complexification of  $\mathfrak{t}$.  Then, $\mathfrak{h}$  is a maximal Abelian subalgebra of $\mathfrak{g}$ such that the linear transformations $(ad(H))_{H\in\mathfrak{h}}$ are simultaneously diagonalizable. Denote by $\mathfrak{h}^*$ the dual space of $\mathfrak{h}$. Let $\alpha\in\mathfrak{h}^*$, and define
$$\mathfrak{g}^{\alpha}=\left\{X\in\mathfrak{g}\ :\ \forall H\in\mathfrak{h},\quad [H,X]=\alpha(H)X \right\}.$$

\begin{defi}
$\alpha \in\mathfrak{h}^*$ is said to be a \textbf{root} of $\mathfrak{g}$ with respect to $\mathfrak{h}$, if $\mathfrak{g}^{\alpha}$ is nonzero, and in this case $\mathfrak{g}^{\alpha}$ is called the corresponding root space. We also denote by   $\widetilde{\Phi} \subset\mathfrak{h}^* $ the set of roots.
\end{defi}
Each root space is of dimension 1. One has that  $\mathfrak{g}^{0}=\mathfrak{h}$ (by the maximal property of $\mathfrak{h}$) and $\mathfrak{g}$ can be decomposed as the following direct sum
$\mathfrak{g}=\mathfrak{h}\bigoplus_{\alpha\in\widetilde{\Phi}}{\mathfrak{g}^{\alpha}},$
called the root space decomposition of   $\mathfrak{g}$. To each  $\alpha\in\widetilde{\Phi}$ we associate the hyperplane $\H_{\alpha}\subset \mathfrak{h}^*$ that is orthogonal to $\alpha$. The set of all hyperplanes $\left\{\ \H_{\alpha}\ :\ \alpha\in\widetilde{\Phi}\right\}$ partition $\mathfrak{h}^*$ into a finite number of open convex regions called the Weyl chambers of $\mathfrak{h}^*$. In what follows, we choose and fix a fundamental  Weyl chamber denoted by $K$.

\begin{defi}
Let $\Phi$ be the set of real roots and 
$\Phi_+=\{\alpha\in\Phi\ :\ \forall \beta\in K\quad \langle\alpha,\beta\rangle\}$ be the set of positive roots. Denote  one-half of the sum of positive roots by
$\rho=\frac{1}{2}\sum_{\alpha\in\Phi _+}\alpha$.
\end{defi}

\subsection{Laplace-Beltrami operator}

The Laplace-Beltrami operator is a generalization to Riemannian manifolds (such as Lie groups) of the usual Laplacian operator. We will denote this operator by $\Delta$. To state the following proposition, note that one may identify the set $\widehat{G}$ with a subset of $\Phi_+$ (see the technical appendix in \cite{MR2446754} for further details on this identification).

\begin{prop}\label{somdpi}
The elements of $\widehat{G}$ are the eigenfunctions of $\Delta$. Let $\pi\in\widehat{G}$. The eigenvalue of  $\pi$ is $\lambda_{\pi}=\norme{\pi+\rho}^2-\norme{\rho}^2,$
where $\| \cdot \|$ is the norm induced by the Killing form. For  $\pi\in\widehat{G}$, one has the following relationship between $d_{\pi}$ and $\lambda_{\pi}$
$$\sum_{\pi\in\widehat{G}:\lambda_{\pi}<T}d_{\pi}^2=WT^{(\dim G)/2}+o(T^{(\dim G)/2}) \mbox{ as } T\to \infty, $$
where
\begin{equation}
W=\frac{\text{vol}G}{\left(2\sqrt{\boldsymbol{\pi}}\right)^{\dim G}\Gamma(1+\frac{1}{2}\dim G)}, \label{eq:W}
\end{equation}
with $\text{vol}G$ denoting the volume of $G$, the bold symbol $\boldsymbol{\pi}$ denoting the number Pi and $\Gamma(.)$ being the classical gamma function.
\end{prop}

\end{appendix}

\bibliographystyle{alpha}
\bibliography{LieGroupDeconvolution}

\end{document}